\documentclass[11pt]{article}
\usepackage{amsmath,amssymb,amsthm}
\usepackage{hyperref} 
\usepackage{tikz-cd} 
\theoremstyle{definition} 
\newtheorem{definition}{Definition}[section]
\newtheorem{remark}[definition]{Remark}
\newtheorem{setting}[definition]{Setting}
\newtheorem{preliminary}[definition]{Preliminary}
\theoremstyle{plain}
\newtheorem{theorem}[definition]{Theorem}
\newtheorem{proposition}[definition]{Proposition}
\newtheorem{corollary}[definition]{Corollary}
\newtheorem{lemma}[definition]{Lemma}
\newcommand{\eq}[1]{\begin{equation}\label{#1}}
	\newcommand{\eeq}{\end{equation}}
\title{Linearisation problem under finite \'etale cover}
\author{Xiaojun Wu}
\date{\today} 
\begin{document} 
\def\A{\mathcal{A}}
\def\cI{\mathcal{I}}
\def\Z{\mathbb{Z}}
\def\Q{\mathbb{Q}}  \def\C{\mathbb{C}}
 \def\R{\mathbb{R}}
 \def\N{\mathbb{N}}
 \def\H{\mathbb{H}}
  \def\P{\mathbb{P}}
 \def\rC{\mathrm{C}}
  \def\d{\partial}
 \def\dbar{{\overline{\partial}}}
\def\dzbar{{\overline{dz}}}
\def \ddbar {\partial \overline{\partial}}
\def\cB{\mathcal{B}}
\def\cD{\mathcal{D}}  \def\cO{\mathcal{O}}
\def\cbarO{\overline{\mathcal{O}}}
\def\D{\mathcal{D}}
\def\cC{\mathcal{C}}
\def\cF{\mathcal{F}}
\def \rank{\mathrm{rank}}
\def \deg{\mathrm{deg}}
\def \tot{\mathrm{Tot}}
\def \id{\mathrm{id}}
\bibliographystyle{plain}
\def \End{\mathrm{End}}
\def \dim{\mathrm{dim}}
\def \div{\mathrm{div}}
\def \ker{\mathrm{Ker}}
\def \im{\mathrm{Im}}
\def \rC{\mathrm{Cone}}
\def \Sym{\mathrm{Sym}}
\def \Hom{\mathrm{Hom}}
\newcommand{\Ub}{\mathcal{U}}
\newcommand{\dcech}{\check{\delta}}
\newcommand{\dcechg}{\delta\!\!\!\check{\delta}}
\newcommand{\lc}{\mathcal{L}}
\newcommand{\ec}{\mathcal{E}}
\newcommand{\var}{\varphi}
\maketitle
\begin{abstract}
In this article, we study the (full or vertical) linearisation problem under a finite Galois étale cover.
As an application, we give sufficient conditions for full linearisation near Hopf manifolds and for vertical or full linearisation near hyperelliptic manifolds.
\end{abstract}
\section{Introduction}
It is  a classical problem to classify neighborhoods $U$ of an embedded compact complex manifold $C$ into complex manifolds $M$ up to biholomorphism fixing $C$ pointwise. In particular, Grauert called ``Das Formale Prinzip''(e.g \cite{kosarew-crelle,Hi81,hwang-annals}) the following problem : Assume there is  a {\it formal equivalence} between a neighborhood $U$ in $M$ and a neighborhood $U'$ in $M'$, in both of which $C$ is holomorphically embedded. Does such a formal equivalence give  rise to a genuine holomorphic equivalence between the (possibly smaller) neighborhoods? Here differential geometry and curvature enter into the play. Indeed, if this normal bundle is {\it negative}, then Grauert~\cite{grauert-embed} and Hironaka-Rossi~\cite{hironaka-rossi} 
proved a rigidity statement: formally equivalent neighborhoods of $C$ in $M$ and $M'$ are actually biholomorphic. In the case the normal bundle is {\it positive}, Griffiths~\cite{griffiths-ext2} proved that there are only
finitely many obstructions to being formally equivalent to $U$. He then
proved, under some assumptions, that if the neighborhood $U'$ is formally
equivalent to $U$, then it is also biholomorphic to it. 

A holomorphic embedding of $C$ into $M$ gives rise to another natural embedding, namely the embedding of $C$ as the zero section in its normal bundle $N_{C/M}$. Equivalence of a neighborhood of $C$ in $M$ with a neighborhood of $C$ in $N_{C/M}$ can be seen as a type of ``linearization". In that case, we shall say that the neighborhood is ``fully linearizable". When the normal bundle is {\it flat},  dynamical systems methods are more appropriate. Indeed, Arnol'd~\cite{arnold-embed}(see also \cite[Chap.~5, sect.~27]{arnold-geometrical}) studied the
embedding of an elliptic curve into a complex surface when the normal bundle has
zero self-intersection number (i.e. {\it flat}). He showed that under a {\it small divisors condition}, the neighborhood is biholomorphic to a (unspecific) neighborhood of the zero section of the normal bundle $N_{C/M}$. This was generalized by Ilyashenko and Pyartly~\cite{ilyashenko-pyartly-embed} to direct product of 1 dimensional tori.

A somehow similar problem, following Ueda \cite{Ued82}, is the existence of a holomorphic foliation on a neighborhood of $C$, having $C$ as a leaf, which can be obtained through a ``vertical linearization" of neighborhoods. 

In recent years, there has been a renewal of interest in these questions (e.g. \cite{loray-moscou,stolo-koike}), in particular in the case of flat normal bundle \cite{GS21,koike-ueda,koike-ueda0, stolo-koike}. 
\paragraph{}
In this article, we study the (full or vertical) linearisation problem under a finite Galois étale cover of complex manifolds.
As an application, we give sufficient conditions for full linearisation near Hopf manifolds and for vertical or full linearisation near hyperelliptic manifolds.

The geometric setting is as follows.
Let $p: \hat{C} \to C$ be a finite étale Galois cover of an $n$-dimensional complex manifold $C$ with Galois group $G$.
Assume that $\hat{C},C$ are connected.
Let $M$ be a complex manifold into which $C$ is holomorphically embedded.
As we shall see below (Lemma \ref{GS24}), there exists a complex manifold $\hat{M}$ into which $\hat{C}$ is holomorphically embedded and a holomorphic map $\hat{M} \to M$ extending $p$.
In the following, we call $\hat{M}$ the induced neighborhood associated with the finite \'etale cover $p$.
We aim to study whether the germ of \(C\) in \(M\) can be linearised (fully or vertically) — that is, whether it is biholomorphic (fully or vertically) to a neighborhood in its normal bundle — via the corresponding problem for \(\hat{C} \subset \hat{M}\).
A sufficient condition is that the linearisation of $\hat{C}$ near $\hat{M}$ is $G$-equivariant, especially if this linearisation is unique.

In general, the first cohomology group provides the obstruction to formal linearisation, while the zeroth cohomology group controls its uniqueness, if a formal linearisation exists. In this article, under the additional assumption that the zeroth cohomology vanishes, we study the linearisation problem.
Note that the assumption on the vanishing of the zeroth cohomology is mild and is satisfied, for instance, if the normal bundle is non-resonant.

Our main result is the following.
\begin{theorem}
\label{ueda_quotient}
 Let $C$ be an $n$-dimensional connected complex manifold embedded in a complex manifold $M$ of dimension $n + d$. Suppose that the normal bundle of $C$ in $M$ admits transition functions that are locally constant Hermitian matrices.
Assume that there exists a connected finite étale Galois cover $p:\hat{C} \to C$ of complex manifolds such that the induced neighborhood of $\hat{C}$ in $\hat{M}$ associated with the finite \'etale cover $p$ is biholomorphic to a neighborhood of the zero section in the induced normal bundle $N_{\hat{C}/\hat{M}}$.
Assume that
$$H^0(\hat{C}, T_{\hat{C}}\otimes \Sym^l N^{*}_{\hat{C}/\hat{M}})=H^1(\hat{C}, T_{\hat{C}} \otimes \Sym^l N^{*}_{\hat{C}/\hat{M}})=0,$$
for all $l\geq 1$,
$$H^0(\hat{C}, N^{}_{\hat{C}/\hat{M}}\otimes \Sym^l N^{*}_{\hat{C}/\hat{M}})=H^1(\hat{C},N^{}_{\hat{C}/\hat{M}}\otimes \Sym^l N^{*}_{\hat{C}/\hat{M}})=0,$$
for all $l\geq 2$.

Then a neighborhood of $C$ in $M$ is biholomorphic to a neighborhood of the zero section in the normal bundle $N_{C/M}$.
\end{theorem} 
As applications, we have the following consequences.
\begin{corollary}
\label{corA}
Let $C$ be a hyperelliptic manifold (that is, a quotient
$C = \hat{C} / G$ of a complex torus $\hat{C}$ by the free action of a finite group $G$ containing no translations) of dimension $n$,
embedded in a complex manifold $M$ of dimension $n + d$. Suppose that the normal bundle of $C$ in $M$ admits transition functions that are locally constant Hermitian matrices.
Let $p: \hat{C} \to C$ be a finite étale Galois cover such that $\hat{C}$ is a complex torus, and assume that the induced normal bundle $N_{\hat{C}/\hat{M}}$ satisfies the Diophantine condition (cf. Definition \ref{dioph_full}).

Then a neighborhood of $C$ in $M$ is biholomorphic to a neighborhood of the zero section in the normal bundle $N_{C/M}$.
\end{corollary}
\begin{corollary}
\label{corB}
Let $C$ be a Hopf manifold of generic or classical type (see, for instance, Definitions \ref{generic} and \ref{classical}) of dimension $n$
embedded in a complex
manifold $M$ of dimension $n + 1$.  Suppose that the normal bundle of $C$ in $M$ is flat.
Let $p:\hat{C} \to C$ be a finite \'etale Galois cover such that $\hat{C}$ is a primary Hopf surface. Assume that the induced normal bundle $N_{\hat{C}/\hat{M}}$ is a non-Hermitian flat line bundle satisfying the irrationality condition in Definition \ref{irrational}.
Then a neighborhood of $C$ in $M$ is biholomorphic to
a neighborhood of the zero section in the normal bundle.
\end{corollary}
Note that in our applications the dynamical condition implies the simultaneous vanishing of the zeroth and first cohomology groups.
We also remark that the Diophantine condition is preserved under finite étale covers.
\paragraph{}
For vertical linearisation, the situation is more complicated.
In general, the bundle-valued holomorphic forms defining a regular holomorphic foliation with $\hat{C}$ as a leaf are not unique.
Our observation is that such a holomorphic form is unique if the terms of degree at least two in its Taylor expansion along the normal directions vanish.
\begin{theorem}
\label{ueda-quotient2}
Let $C$ be an $n$-dimensional connected complex manifold embedded in a complex manifold $M$ of dimension $n+1$.
Assume that $H^0(C, \cO_C)=\C$.
Suppose that the normal \emph{vector bundle} $N_{C/M}$ of $C$ in $M$ admits transition functions that are locally constant Hermitian matrices.
Assume that there exists a connected finite étale Galois cover
$p: \hat{C} \to C$ of complex manifolds such that the induced neighborhood of $\hat{C}$ in $\hat{M}$
associated with the finite \'etale cover $p$ admits vertical linearisation (see Setting \ref{set}).
Assume that the short exact sequence
$$
0 \;\to\; N^*_{\hat{C}/\hat{M}} \;\to\; \Omega^1_{\hat{M}}|_{\hat{C}}
   \;\to\; \Omega^1_{\hat{C}} \;\to\; 0 ,
$$ splits and that
$$
H^0(\hat{C},N^{-k}_{\hat{C}/\hat{M}})=H^1(\hat{C},N^{-k}_{\hat{C}/\hat{M}})=0,
$$$$
H^0(\hat{C},\Omega^1_{\hat{C}} \otimes N^{-k}_{\hat{C}/\hat{M}})=H^1(\hat{C},\Omega^1_{\hat{C}} \otimes N^{-k}_{\hat{C}/\hat{M}})=0,
$$
for all $k \geq 1$.
Assume furthermore that
$$
H^0\!\bigl(\hat{C},\, \Omega^1_{\hat{C}} \otimes N_{\hat{C}/\hat{M}} \bigr) \;=\;
0.
$$
Then a neighborhood of $C$ in $M$ admits a regular holomorphic foliation with $C$ as a leaf.
\end{theorem}
The higher-codimension case of Theorem \ref{ueda-quotient2} is somewhat more difficult to state, but we can give an answer in the case of hyperelliptic manifolds in arbitrary codimension.
\begin{corollary}
\label{corC}
Let $C$ be a hyperelliptic manifold of dimension $n$,
embedded in a complex manifold $M$ of dimension $n + d$. Suppose that the normal bundle of $C$ in $M$ admits transition functions that are locally constant Hermitian matrices.
Let $p: \hat{C} \to C$ be a finite étale Galois cover such that $\hat{C}$ is a complex torus, and assume that the induced normal bundle $N_{\hat{C}/\hat{M}}$ satisfies the strongly vertically Diophantine condition (cf. Definition \ref{dioph}).

Then a neighborhood of $C$ in $M$ admits a regular holomorphic foliation with $C$ as a leaf.
\end{corollary}
The paper is organised as follows. In Section~2, we introduce the general setting and recall several basic notions and results that will be used throughout the paper. Section~3 is devoted to the study of the full linearisation problem under a finite Galois étale cover, and we prove Theorem~\ref{ueda_quotient} together with its corollaries. In Subsection~4.1, we collect some background material on holomorphic foliations that will be used in the later arguments. Finally, Subsection~4.2 is devoted to the study of the vertical linearisation problem under a finite Galois étale cover, and we prove Theorem~\ref{ueda-quotient2} and its corollaries.
\paragraph{}
\textbf{Acknowledgements.} I would like to thank my previous postdoctoral advisor, Professor Laurent Stolovitch, for reading the draft of this work and providing several important suggestions that substantially improved it. This research was supported by the JSPS Postdoctoral Fellowships for Research in Japan (Standard). I am also grateful to Osaka Metropolitan University and the Osaka Central Advanced Mathematical Institute (OCAMI) for providing an excellent research environment.
\section{Terminology}
In this section, we recall the terminology and basic definitions from \cite{GS21} used in this article.

We begin by recalling the following result from \cite{stolo-gong-tori}. As it will be used repeatedly, we provide a brief sketch of the proof for completeness.
\begin{lemma} \label{GS24} Let \( C \) be a complex manifold and let \( p : \hat{C} \to C \) be a holomorphic covering. Suppose that \( (M, C) \) is a holomorphic neighborhood of \( C \). Then there exists  a holomorphic map \( \hat{M} \to M \) defined in a neighborhood of \( \hat{C} \), sending \( \hat{C} \) into \( C \) and restricting to \( p \) on \( \hat{C} \). Moreover, this map is a holomorphic covering. We call $\hat{M}$ the induced neighborhood associated with the finite \'etale cover $p$. \end{lemma} \begin{proof} Since \(C\) is a complex submanifold of \(M\), it is also a smooth submanifold of the underlying real manifold. By the tubular neighborhood theorem, after possibly shrinking \(M\) around \(C\), there exists a smooth retraction \( \pi : M \to C \) such that \( \pi|_C = \mathrm{id}_C \). We define the smooth fiber product \[ \begin{tikzcd} \hat{M}:=M \times_C \hat{C} \arrow[r] \arrow[d, "\hat{p}"] & \hat{C} \arrow[d, "p"] \\ M \arrow[r, "\pi"] & C. \end{tikzcd} \] Then \(\hat{M}\) is a smooth manifold and \(\hat{p} : \hat{M} \to M\) is a smooth covering whose restriction to \(\hat{C}\) coincides with \(p\). We now endow \(\hat{M}\) with a complex structure. Since \(\hat{p}\) is a local diffeomorphism, the differential induces an isomorphism of vector bundles \( T_{\hat{M}} \simeq \hat{p}^* T_M \). Let \(J\) be the complex structure on \(M\). Pulling back \(J\) gives an almost complex structure \[ \hat{J} := \hat{p}^* J \] on \(\hat{M}\). We claim that \(\hat{J}\) is integrable. Indeed, the Nijenhuis tensor behaves functorially under pullback by local diffeomorphisms. Since the Nijenhuis tensor of \(J\) vanishes on \(M\), the Nijenhuis tensor of \(\hat{J}\) also vanishes. Therefore, by the Newlander--Nirenberg theorem (cf. e.g. \cite[Chapter VIII, Theorem (11.8)]{agbook}), \(\hat{J}\) is integrable and defines a complex structure on \(\hat{M}\). By construction, \(\hat{p}\) is holomorphic. Since it is locally biholomorphic, it is a holomorphic covering. \end{proof}
The following remark will be used frequently. Unless otherwise specified, we always endow the finite \'etale cover of the ambient manifold with this holomorphic group action.
\begin{remark} 
\label{rem1}
If \(p : \hat{C} \to C\) is a Galois covering of complex manifolds with Galois group \(G\), then \(\hat{p} : \hat{M} \to M\) is also a Galois covering of complex manifolds with the same Galois group \(G\). \end{remark}
\begin{setting}
\label{set}
A neighborhood $V$ of an embedded complex  manifold  $C_n$ in    $M_{n+d}$ has local holomorphic charts $(h_j,v_j)=\Phi_j$ mapping $V_j$ onto   $\hat V_j$ in $\C^{n+d}$ with $n=\dim \; C$. Here $\cup V_j$ is a neighborhood of $C$ and  $U_j:=V_j\cap C$ is defined by $\{v_j=0\}$. The classification of the germs of neighborhoods of $C$ is then the classification of transition functions $\Phi_{kj}:=\Phi_k \circ \Phi_{j}^{-1}$ under
holomorphic conjugacy $F_k^{-1} \circ \Phi_{kj} \circ F_j$.
To such an embedding, one can associate the normal bundle $N_{C/M}$
of $C$ in $M$, which has
the transition matrices $g_{kj}(p)$, $p\in U_k\cap U_j$. To this embedding one can associate another natural embedding, namely the embedding of $C$ as the zero section of $N_{C/M}$.
Under a mild assumption, this embedding $(N_{C/M}, C)$ naturally provides a first-order approximation of $(M,C)$.
Let $\var_j=\Phi_j|_{U_j}$ and let $\var_{kj}=\var_k\circ \var_j^{-1}$  be the transition functions of $C$.
To have a neighborhood of $C$ in $M$ equivalent to a neighborhood of the zero
section in $N_{C/M}$ is equivalent to 
seeking
 $F_j$ such that $\hat\Phi_{kj}=F_k^{-1}\circ \Phi_{kj} \circ F_j$ are of the form $N_{kj}(h_j,v_j)=(\var_{kj}(h_j),t_{kj}(h_j)v_j)$
with $t_{kj}(h_j)=g_{kj}$, the latter being  %
regarded as the transition functions of a neighborhood of the zero
section of $N_{C/M}$. 
We call this process the ``full linearization" of the neighborhood.

In Ueda's theory (\cite{Ued82}), we are interested in the existence of a regular holomorphic foliation in a germ of a neighborhood of $C$ in a complex manifold, with $C$ as a compact leaf.  We refer to it as a ``horizontal foliation".
The
 ``horizontal foliation" will be obtained as a consequence of a ``vertical linearization" of the neighborhood, which amounts to 
 seeking  $F_j$ such that $\hat\Phi_{kj}=(\var_{kj}(h_j)+\hat\phi_{kj}^h(h_j,v_j), t_{kj}(h_j)v_j)$.
In the following, we refer to the problem of finding such a ``horizontal foliation'' as Ueda's problem.
\end{setting}

 \begin{definition}(\cite[Definition 2.1]{GS21})
 \label{def2.1}
 We call $\hat M_C$ an (admissible and splitting)   formal neighborhood of $C$ if   there are holomorphic coordinate charts $
 \var_j$ on $U_j$ where $\{U_j\}$ is a covering of $C$ and there are formal power series
$$
(z_j,w_j)=\hat\Phi_j(p,w):=(
\var_j(p),
t_j(p)w)+\sum_{|Q|\geq2}\Phi_{j,Q}(p)w^Q,
$$
where $\Phi_{j,Q}$ are holomorphic functions in $U_j$ and each $t_j$ is an invertible holomorphic  $d\times d$ matrix on $U_j$.  Note that the formal transition functions $\hat\Phi_{kj}=\hat \Phi_k \circ \hat\Phi_j^{-1}$ have the form
\eq{}\nonumber
\hat\Phi_{kj}(z_j,w_j)=(
\var_{kj}(z_j),t_{kj}(z_j)w_j)+\sum_{|Q|>1}
\hat \Phi_{kj,Q}(z_j)w_j^Q, \quad z_j\in
\var_j(U_j\cap U_k).
\eeq
\begin{enumerate}
\item
When all $\Phi_j$ are
holomorphic, the formal neighborhood $\hat M_C$ is called
the germ of a  (holomorphic)
 neighborhood of $C$.
\item $\hat M_C$ is called a 
 linear neighborhood of $C$ if additionally
\eq{initial000}   
\hat\Phi_{kj}(z_j,v_j)=(
\var_{kj}(z_j),t_{kj}(z_j)v_j)
\eeq
and each
 $t_{kj}$ is
 an invertible holomorphic matrix in $U_k\cap U_j$. The terminology is meaningful since the $\hat\Phi_{kj}$ can be realized as the transition functions of a holomorphic vector bundle over $C$,  namely the normal bundle of $C$ in $M$.
 \end{enumerate}
 \end{definition}
\begin{definition}(\cite[Definition 2.5]{GS21})
We say that a formal  neighborhood  $\{\Phi_{kj}\}$ of $C$ is
  equivalent to a neighborhood $\{\hat\Phi_{kj}\}$ of $C$ in $M$ via a formal mapping  $F$
that is tangent to the identity and preserves the splitting of $T_M|_C$, if there are formal maps $F_j(z_j)=(z_j,w_j)+\sum_{|Q|>1} F_{j,Q}(z_j)w_j^Q$ such that $F_{j,Q}(z_j)$ are holomorphic functions in $  U_j$ and as power series in $w_j$
$$
F_k \circ \hat\Phi_{kj}(z_j,w_j)=\Phi_{kj} \circ F_j(z_j,w_j).
$$
 We take $F=\hat\Phi_j^{-1} \circ F_j \circ \Phi_j$, which is well-defined
 when $\Phi_{kj}=\Phi_k \circ \Phi_j^{-1}$ and $\hat\Phi_{kj}=\hat\Phi_k \circ \hat\Phi_j^{-1}$.
\end{definition} 
An equivalent definition can be given as follows.
\begin{definition}
\label{formal-neigh}
Let $C$ be a complex submanifold of a complex manifold $M$, with inclusion $i: C \hookrightarrow M$.
For each $k \ge 0$, let $\mathcal{I}_C$ denote the ideal sheaf of $C$ in $\mathcal{O}_M$. The \emph{$k$-th infinitesimal neighborhood} of $C$ in $M$ is defined as the ringed space
\[
(C, \mathcal{O}_M / \mathcal{I}_C^{k+1}).
\]

The \emph{formal neighborhood} of $C$ in $M$ is defined as the formal scheme
\[
\hat{M}_C := \varprojlim_{k \ge 0} (C, \mathcal{O}_M / \mathcal{I}_C^{k+1}),
\]
where the limit is taken in the category of ringed spaces.

Equivalently, the structure sheaf of the formal neighborhood is the completion of $\mathcal{O}_M$ along $C$:
\[
\mathcal{O}_{\hat{M}_C} := \varprojlim_{k \ge 0} \mathcal{O}_M / \mathcal{I}_C^{k+1}.
\]
\end{definition}
\begin{definition}
\label{germ}
Let \( M \) be a complex manifold and \( C \subset M \) a complex submanifold with inclusion $i: C \hookrightarrow M$. A \emph{germ of holomorphic neighborhood of \( C \)} is defined as follows. Two pairs \( (M_1, C) \) and \( (M_2, C) \), where \( C \subset M_i \) is a complex submanifold, are said to be \emph{equivalent as neighborhoods of \( C \)} if there exist open neighborhoods \( U_i \subset M_i \) of \( C \) respectively and a biholomorphism \( \Phi : U_1 \longrightarrow U_2 \) such that \( \Phi|_{C} = \mathrm{id}_C . \) The \emph{germ of holomorphic neighborhood of \( C \) in $M$} is the equivalence class of such pairs under this relation containing $(M,C)$. 
Equivalently, it is the inductive limit of all open neighborhoods of \( C \) in \( M \), taken in the category of holomorphic ringed spaces. 
\end{definition}
\begin{proposition} \label{inj} Let \( (M_1, C_1) \), \( (M_2, C_2) \) be germs of holomorphic neighborhoods. Then there is a canonical injective map \[ \Hom((M_1, C_1), (M_2, C_2)) \to \Hom((\hat{M_1}_{C_1},C_1),(\hat{M_2}_{C_2},C_2)). \]
Here
$\Hom((M_1,C_1),(M_2,C_2))$ denotes the set of holomorphic morphisms of germs from $(M_1,C_1)$ to $(M_2,C_2)$. Namely, an element is represented by a holomorphic map $f : U_1 \to U_2$, where $U_i$ is a sufficiently small neighborhood of $C_i$ in $M_i$, such that $f(C_1)\subset C_2$, modulo shrinking $U_1,U_2$.
Accordingly, $\Hom((\hat{M_1}_{C_1},C_1),(\hat{M_2}_{C_2},C_2))$ denotes the set of morphisms of formal neighborhoods, namely morphisms of locally ringed spaces compatible with the formal structures.
 \end{proposition} \begin{proof} Let \(F \in \Hom((M_1, C_1), (M_2, C_2))\). Let \(f := F|_{C_1} : C_1 \to C_2\). Since \(F(C_1) \subset C_2\), we have \[ F^*(\mathcal I_{C_2}) \subset \mathcal I_{C_1}. \] Hence for every \(k \ge 0\), \[ F^*(\mathcal I_{C_2}^{k+1}) \subset \mathcal I_{C_1}^{k+1}. \] Therefore, for every \(p \in C_1\), the pullback induces a morphism of local rings \[ F^* : \mathcal O_{M_2,F(p)} / \mathcal I_{C_2,F(p)}^{k+1} \to \mathcal O_{M_1,p} / \mathcal I_{C_1,p}^{k+1}. \] Passing to the inverse limit, we obtain \[ \varprojlim_k \mathcal O_{M_2,F(p)} / \mathcal I_{C_2,F(p)}^{k+1} \to \varprojlim_k \mathcal O_{M_1,p} / \mathcal I_{C_1,p}^{k+1}. \] These maps are compatible with restriction and define a morphism of sheaves \[ f^{-1}\hat{\mathcal O}_{M_2,C_2} \to \hat{\mathcal O}_{M_1,C_1}. \] Thus we obtain a morphism of formal neighborhoods \[ \hat{F} \in \Hom((\hat{M_1}_{C_1},C_1),(\hat{M_2}_{C_2},C_2)). \] 
We now prove that this map is injective.
 Let \(F_1, F_2\) be two morphisms inducing the same formal morphism. Then they induce the same morphism on the zeroth infinitesimal neighborhood, hence $ F_1|_{C_1} = F_2|_{C_1}. $ Let \(p \in C_1\). For any germ \(f\) of holomorphic function on \(M_2\) at the point \(F_1(p)=F_2(p)\), the equality of the formal pullbacks implies \[ F_1^*(f) \equiv F_2^*(f) \quad \text{mod } \mathcal I_{C_1,p}^{k+1} \] for all \(k\). Hence \[ F_1^*(f) - F_2^*(f) \in \bigcap_{k \ge 1} \mathcal I_{C_1,p}^k. \] Since the local ring is Noetherian, Krull's intersection theorem (cf. e.g. \cite[Chapter IX, Lemma (1.3)]{agbook}) gives $ \bigcap_{k \ge 1} \mathcal I_{C_1,p}^k = 0. $ Therefore \(F_1^*(f) = F_2^*(f)\) for every germ \(f\) of holomorphic function on \(M_2\) at the point \(F_1(p)=F_2(p)\). Let \((z_1,\dots,z_n)\) be a system of local holomorphic coordinates of \(M_2\) centered at this point. Then for each \(i\) we have \[ F_1^*(z_i) = F_2^*(z_i) \] as germs of holomorphic functions at \(p\). This means that the coordinate expressions of the two maps coincide near \(p\). Hence the two holomorphic maps $ F_1, F_2 : (M_1,p) \to (M_2,F_1(p)) $ coincide as germs. Since this holds for every \(p \in C_1\), we conclude that \(F_1\) and \(F_2\) coincide as germs of holomorphic maps along \(C_1\). This proves the injectivity. ` \end{proof}
\begin{remark}
In general, the map 
$$\Hom((M_1, C_1), (M_2, C_2)) \to \Hom((\hat{M_1}_{C_1},C_1),(\hat{M_2}_{C_2},C_2))$$ 
may fail to be surjective due to the small divisor phenomenon.
\end{remark}
\section{Full linearisation}
To begin, we study the behavior of cohomology under a finite étale Galois cover. 
For the convenience of a general audience, we recall the relevant construction.
We refer to the book \cite{Mum70} for more details.
\begin{preliminary}
\label{mumford}
Let $p:\hat{C} \to C$ be a finite étale Galois cover of an $n$-dimensional complex manifold $C$ with Galois group $G$.

For any complex character $\chi \in \hat{G}$ (that is, a one-dimensional complex representation of \(G\)), define
\[
L_\chi := \{\alpha \in p_* \mathcal{O}_{\hat{C}} \mid \forall g \in G, \, g \cdot \alpha = \chi(g) \alpha \}.
\]

By Maschke's theorem (cf. e.g. \cite[Theorem 2]{Ser77}), every representation of the finite group $G$ is a direct sum of irreducible complex representations.
One can check that the completely reducible decomposition of the representation $G \to p_* \mathcal{O}_{\hat{C}}(U)$ for any open set $U \subset C$ defines a direct sum decomposition of the sheaf $p_* \mathcal{O}_{\hat{C}}$.
As this decomposition is defined by holomorphic equations, $L_\chi$ is a coherent subsheaf of $p_* \mathcal{O}_{\hat{C}}$.

If $G$ is Abelian, we thus have a decomposition into holomorphic line bundles:
\[
p_* \mathcal{O}_{\hat{C}} = \bigoplus_{\chi \in \hat{G}} L_\chi,
\]
(see e.g., \cite[Page 72]{Mum70}).
Note that
the component corresponding to the trivial character, $L_1$, is always $\mathcal{O}_C$ (even if $G$ is not Abelian) (see \cite[Corollary, Page 72]{Mum70}).
The Abelian assumption is used in the literature solely to ensure that every irreducible representation is one-dimensional. 
When $G$ is Abelian, the $L_\chi$'s are line bundles, which makes it easier to compute their cohomologies.
However, the component corresponding to the trivial character is always one-dimensional, independently of whether the group is Abelian. This is sufficient for our purposes.

Let $E$ be a vector bundle over $C$. Then the Leray spectral sequence together with the projection formula implies, for all $i$,
\[
H^i(\hat{C}, p^* E) = H^i(C, E \otimes_{\mathcal{O}_C} p_* \mathcal{O}_{\hat{C}}).
\]

In particular, if
\[
H^i(\hat{C}, p^* E) = 0,
\]
then
\[
H^i(C, E) = 0
\]
by considering the component corresponding to the trivial character (that is $\cO_C \subset p_* \cO_{\hat{C}}$).


In the following, we will only use this last conclusion, without assuming that the Galois group is Abelian. We will see the examples when $G$ is not necessarily Abelian (see Remark \ref{hyperell}).
\end{preliminary}
\paragraph{}
Let us study the full linearisation problem.

\textit{Proof of Theorem \ref{ueda_quotient}:}
Note that the vanishing of cohomology on $\hat{C}$ implies the corresponding vanishing on $C$ as discussed in Preliminary \ref{mumford}.
We also note that the Hermitian flatness of \( N^*_{C/M} \) implies the Hermitian flatness of \( N^*_{\hat{C}/\hat{M}} \). Indeed, the differentials of \( p \) and $\hat{p}$ induce isomorphisms \( T_{\hat{C}} \to p^* T_C \) and \( T_{\hat{M}} \to \hat{p}^* T_M \). It follows that \( p^* N^*_{C/M} = N^*_{\hat{C}/\hat{M}} \). From the representation-theoretic viewpoint, the normal bundle \( N^*_{\hat{C}/\hat{M}} \) corresponds to the representation \( \pi_1(\hat{C}) \to \pi_1(C) \to GL(d,\C) \), where \( \pi_1(C) \to GL(d,\C) \) is the representation associated with \( N^*_{C/M} \). Therefore, Hermitian flatness is preserved under pullback.

The vanishing of
$H^1(C, T_{C}\otimes  N^{*}_{{C}/{M}})$
implies a holomorphic splitting$$T_M|_C \simeq T_C \oplus N_{C/M}.$$
The cohomology vanishing conditions further imply the existence of a unique formal holomorphic isomorphism
$$\Phi: (M,C) \to (N_{C/M},C)$$
which is tangent to the identity and preserves the splitting of $T_M|_C$ (by \cite[Lemma 2.10, Proposition 2.14]{GS21}). More precisely, the vanishing of $H^1$ ensures the existence of $\Phi$ \cite[Lemma 2.10]{GS21}, while the vanishing of $H^0$ guarantees its uniqueness \cite[Lemma 2.10, Proposition 2.14]{GS21}.

Denote
$$G:=Gal(\hat{C}/C).$$
The conclusion follows if we can show that the biholomorphism $\Psi$ (provided by the assumptions) between a neighborhood of $\hat{C}$ in $\hat{M}$ and a neighborhood of the zero section in the induced normal bundle $N_{\hat{C}/\hat{M}}$ is $G$-equivariant.

However, $\Phi$ induces a $G$-equivariant formal isomorphism between a neighborhood of $\hat{C}$ in $\hat{M}$ and a neighborhood of the zero section in $N_{\hat{C}/\hat{M}}$. 
More precisely, by Remark \ref{rem1}, $(\hat{M}, \hat{C}) \to (M,C)$ is a Galois cover.
By Proposition \ref{inj}, it induces a Galois cover between a formal neighborhood of $\hat{C}$ in $\hat{M}$ and a formal neighborhood of $C$ in $M$. 
Similarly, we have a Galois cover between a formal neighborhood of the zero section in $N_{\hat{C}/\hat{M}}$ and a formal neighborhood of the zero section in $N_{C/M}$. 
For any point $x \in \hat{C}$, we have isomorphisms between the following: stalk of formal neighborhood of $\hat{C}$ in $\hat{M}$ at $x$,
stalk of formal neighborhood of $C$ in $M$ at $p(x)$,
stalk of formal neighborhood of $C$ in $N_{C/M}$ at $p(x)$,
stalk of formal neighborhood of $\hat{C}$ in $N_{\hat{C}/\hat{M}}$ at $x$.
These natural identifications induce a $G$-equivariant formal isomorphism between a neighborhood of $\hat{C}$ in $\hat{M}$ and a neighborhood of the zero section in $N_{\hat{C}/\hat{M}}$. 

By uniqueness of formal isomorphism, this induced formal isomorphism coincides with $\Psi$, which implies that $\Psi$ is $G$-equivariant by Proposition \ref{inj}.
\qed
\paragraph{}
Since the linearization problem under suitable Diophantine conditions has recently been solved for complex tori in \cite{stolo-gong-tori} and for Hopf manifolds in \cite{SW25}, while the vertical linearization problem under suitable strongly vertically Diophantine conditions has recently been solved for complex tori in \cite{SW24}, we apply Theorem \ref{ueda_quotient} to finite \'etale quotients of these two families of manifolds.

In the following, we discuss these different families of manifolds separately.

For the convenience of the reader, we recall the (strongly vertically) Diophantine condition used in \cite{stolo-gong-tori} and \cite{SW24}. 

Let $C=\C^n/\Lambda$ be a complex torus with lattice \[ \Lambda=\sum_{i=1}^{2n}\Z e_i, \] where $e_1,\ldots,e_n$ form the standard basis of $\C^n$.
Define, for each \(1 \leq j \leq n\),
\[
e_{n+j} =: \tau_j = (\tau_{j1}, \ldots, \tau_{jn}) \in \mathbb{C}^n.
\]
Consider the covering 
$$\pi \colon \tilde{C}:=\C^{n}/{\sum_{i=1}^{n}\Z e_{i}}\simeq (\C^*)^n \to C.$$ Let $(M,C)$ be a neighborhood of $C$, and assume that the normal bundle $N_{C/M}$ admits transition functions given by locally constant Hermitian matrices. Then, by \cite[Proposition 4.3]{stolo-gong-tori}, the germ $(M,C)$ is holomorphically equivalent to the quotient of an open neighborhood of $\tilde{C}$ in $\widetilde{N}_{\tilde{C}/\tilde{M}}$ by the Deck transformation group of $\tilde{C}$. Here $\widetilde{N}_{\tilde{C}/\tilde{M}}$ denotes the induced neighborhood of $\tilde{C}$ obtained by pulling back $(N_{C/M},C)$ via $\pi$.
The deck transformations of $\widetilde{N}_{\tilde{C}/\tilde{M}}$ are generated by $n$ biholomorphisms $\hat \tau_1,\cdots,\hat \tau_n$ that preserve $\tilde{C}$.
\begin{equation}
\hat \tau_j (h,v)= (T_j h, M_j v), M_j := \mathrm{diag}(\mu_{j,1}, \cdots , \mu_{j,d})
\end{equation}
with $(h,v) \in (\C^*)^n \times \C^d $  the total space of $\widetilde{N}_{\tilde{C}/\tilde{M}}$, $T_j$ being defined by :
\begin{equation}
h=(e^{2 \pi \sqrt{-1} z_1}, \cdots, e^{2 \pi \sqrt{-1} z_n}), T_j:=\mathrm{diag}(\lambda_{j,1}, \cdots , \lambda_{j,d}), \lambda_{j,k}:=e^{2 \pi \sqrt{-1} \tau_{jk}}.
\end{equation} 
Now let us recall the following Diophantine condition introduced in \cite{stolo-gong-tori}.
\begin{definition}
(\cite[Definition 4.5]{stolo-gong-tori})
\label{dioph_full}
	The pullback 
	normal bundle $\widetilde{N}_{\tilde{C}/\tilde{M}}$  is said to be  {Diophantine} if for all $(Q,P)\in \mathbb{N}^d\times \Z^{n}$,   $|Q|>1$ and all $l\in \{1,\ldots, n\}$,  $j\in  \{1,\ldots ,d\}$, $i \in \{1,\ldots, n\}$,
	\begin{equation}
		\max_{l \in \{1,\ldots, n\}}  \left |\lambda_l^P 
		\mu_{l}^Q-\lambda_{{l},i}\right |   >  \frac{D}{(|P|+|Q|)^{\tau}}, \label{dh_full}
	\end{equation}
	\begin{equation}
		\max_{l \in \{1,\ldots, n\}}  \left |\lambda_l^P 
		\mu_{l}^Q-\mu_{{l},j}\right |   >  \frac{D}{(|P|+|Q|)^{\tau}}\label{dv_full}
	\end{equation}
	for some $D>0$, $\tau>0$ (independent of $P,Q$).
\end{definition}
We will also use the following variant, called the strongly vertically Diophantine condition, introduced in \cite{SW24}.
\begin{definition}
(\cite[Definition 2.3]{SW24})
\label{dioph}
	The pullback 
	normal bundle $\widetilde{N}_{\tilde{C}/\tilde{M}}$  is said to be  {strongly vertically Diophantine} if for all $(Q,P)\in \mathbb{N}^d\times \Z^{n}$,   $|Q|>1$ and all   $j=1,\ldots ,d$, $l\in \{1,\ldots, n\}$
	\begin{equation}
		\left |\lambda_l^P 
		\mu_{l}^Q-\mu_{{l},j}\right |   >  \frac{D}{(|P|+|Q|)^{\tau}}\label{dv}
	\end{equation}
	for some $D>0$, $\tau>0$ (independent of $P,Q$).
\end{definition}

\paragraph{}
\textit{Proof of Corollary \ref{corA}:}
Let us show that the Diophantine condition \ref{dioph_full} and the following Lemma \ref{cohom_torus} implies that
$$H^0(\hat{C}, T_{\hat{C}}\otimes \Sym^l N^{*}_{\hat{C}/\hat{M}})=H^1(\hat{C}, T_{\hat{C}} \otimes \Sym^l N^{*}_{\hat{C}/\hat{M}})=0,$$
for all $l\geq 1$,
$$H^0(\hat{C}, N^{}_{\hat{C}/\hat{M}}\otimes \Sym^l N^{*}_{\hat{C}/\hat{M}})=H^1(\hat{C},N^{}_{\hat{C}/\hat{M}}\otimes \Sym^l N^{*}_{\hat{C}/\hat{M}})=0$$
for all $l\geq 2$.
Indeed, since the fundamental group of $\hat{C}$ is free Abelian of finite rank, its normal bundle as Hermitian flat vector bundles are direct sum of Hermitian flat line bundles.
The Diophantine condition \ref{dioph_full} implies that every direct summand of $\Sym^l N^{*}_{\hat{C}/\hat{M}}$ for all $l \geq 1$, as well as every direct summand of $N_{\hat{C}/\hat{M}} \otimes \Sym^l N^{*}_{\hat{C}/\hat{M}}$ for all $l \geq 2$, is non-torsion.
The conclusion follows from Theorem \ref{ueda_quotient} and \cite[Theorem 1.1]{stolo-gong-tori}.
\qed
\paragraph{}
The following lemma is well known to experts. Since we have not found a precise reference in the literature, we include a proof for the reader’s convenience. \begin{lemma} \label{cohom_torus} Let $T$ be a complex torus of dimension $n$, and let $L$ be a non-torsion Hermitian flat line bundle on $T$, that is, $L^{\otimes m}$ is non-trivial for every integer $m \geq 1$. Then \[ H^0(T,L)=H^1(T,L)=0. \] \end{lemma}

\begin{proof}
Let $\C_\chi$ be the rank one local system corresponding to the character $\chi$ associated with the flat line bundle $L$.
The standard argument in Hodge theory for Hermitian flat bundles shows that for every $k$, \[ H^k(T,\C_\chi) \simeq \bigoplus_{p+q=k} H^{p,q}(T,L). \] In particular, it is enough to prove that \[ H^0(T,\C_\chi)=H^1(T,\C_\chi)=0. \] We now compute the cohomology of $\C_\chi$ by topological methods. Since $T$ is a complex torus, it is diffeomorphic to $(S^1)^{2n}$, and in particular $\pi_1(T)\simeq \Z^{2n}$. The character $\chi$ is thus a homomorphism $\Z^{2n} \to \C^*$, which can be written as a product of characters \[ \chi=\prod_{i=1}^{2n}\chi_i, \] where each $\chi_i$ is a character of $\Z$. The assumption that $L$ is non-torsion means that $\chi$ is nontrivial, hence there exists an index $i_0$ such that $\chi_{i_0}$ is nontrivial. 

We first show that $ H^0(S^1,\C_{\chi_{i_0}})=0. $ Indeed, $H^0(S^1,\C_{\chi_{i_0}})$ is the space of flat sections of the corresponding local system. Such a section is determined by a vector $v \in \C$ that is invariant under the monodromy, namely $\chi_{i_0}(1)v=v$. Since $\chi_{i_0}$ is nontrivial, this implies $v=0$, and therefore $H^0(S^1,\C_{\chi_{i_0}})=0$. Next, we compute $H^1(S^1,\C_{\chi_{i_0}})$ using Poincar\'e duality. Since $\C_{\chi_{i_0}}$ is a rank one local system, its dual local system is $\C_{\chi_{i_0}^{-1}}$. Poincar\'e-Verdier duality for the compact oriented manifold $S^1$ (see e.g. \cite[(3.1.1)]{KS90}) yields a natural isomorphism \[ H^1(S^1,\C_{\chi_{i_0}}) \simeq H^0(S^1,\C_{\chi_{i_0}^{-1}})^{\vee}=0. \] 
Finally, we apply the K\"unneth formula for cohomology with coefficients in local systems on products. Since $T \simeq (S^1)^{2n}$ and $\C_\chi$ decomposes as the exterior tensor product of the local systems $\C_{\chi_i}$, the K\"unneth formula gives \[ H^k(T,\C_\chi) \simeq \bigoplus_{k_1+\cdots+k_{2n}=k} \bigotimes_{i=1}^{2n} H^{k_i}(S^1,\C_{\chi_i}). \] Because one factor, namely $H^0(S^1,\C_{\chi_{i_0}})$ and $H^1(S^1,\C_{\chi_{i_0}})$, vanishes, all summands corresponding to $k=0$ and $k=1$ vanish. Hence \[ H^0(T,\C_\chi)=H^1(T,\C_\chi)=0. \] This completes the proof.
\end{proof}
\begin{remark}
\label{hyperell}
In dimension $n=2$, a hyperelliptic manifold $T/G$ is necessarily projective, and
$G$ is necessarily cyclic. In dimension $n=3$, the only examples with
non-Abelian $G$ occur when $G$ is the dihedral group of order $8$, as classified in \cite{UY76} and \cite{CD20}.
In particular, note that $G$ need not be cyclic or Abelian.
\end{remark}

Recall that the cohomology of Hopf manifolds of generic or classical type was computed by Mall in \cite{Mall91}. Cohomology computations for certain non-primary Hopf manifolds were later obtained in \cite{LZ04, GZ06}.
\begin{definition}
\label{generic}
	A Hopf manifold $X$ is called of generic type if it is generated by a
	contraction of the type $\varphi: (z_1
	,\cdots,z_n) \to (\alpha_1 z_1, \cdots, \alpha_n z_n)$ with $0<|\alpha_1 | \cdots
	< |\alpha_n | <1$ such that $(\C^n \setminus \{0\})/\langle \varphi \rangle$, and there are no relations except trivial ones between the $\alpha_i$ of the
	form
	$$ \prod_{i \in A} \alpha_i^{r_i} = \prod_{j \in \{1, \cdots, n\} \setminus A} \alpha_j^{r_j}, r_i \in \N
	$$
	for any $A \subset \{1, \cdots, n\}$.
\end{definition}
\begin{definition}
\label{classical}
				A Hopf manifold $X$ (of dimension $>3$) is called of classical type if it is generated by a
				contraction of the type $\varphi: (z_1
				,\cdots,z_n) \to (\alpha z_1, \cdots, \alpha z_n)$ with $0
				< |\alpha | <1$.
\end{definition}
In the following, we consider a non-Hermitian flat line bundle over a Hopf manifold of generic or classical type, corresponding to $\beta \in \C^*$ satisfying the following irrational condition, in order to apply \cite[Theorem 3.3]{SW25}. 
\begin{definition}
\label{irrational}
Let $X$ be a Hopf manifold of generic or classical type.
Let $L$ be a non-Hermitian flat line bundle over $X$ corresponding to $\beta \in \mathbb{C}^*$.

In the classical case, we say that $L$ is irrational if $\beta \notin \Delta_{\alpha}$, that is, if $\beta$ does not belong to the free commutative subgroup of $\mathbb{C}^*$ generated by $\alpha$.

In the generic case, we say that $L$ is irrational if $\beta \alpha_i^{-1} \notin \Delta_{\alpha_1,\ldots,\alpha_n}$ for every $i$.
\end{definition}

A similar condition in the case of Hopf surfaces was studied in \cite{Tsu84}.

Theorem \ref{ueda_quotient} also applies to secondary Hopf manifolds.
Recall that a secondary Hopf manifold is a complex manifold admitting a (primary) Hopf manifold as a finite \'etale covering.
In particular, we can study Hopf manifolds with non-Abelian fundamental groups, whose construction can be found, for example, in \cite{Ka75} and \cite{Hae85}.
\paragraph{}
\textit{Proof of Corollary \ref{corB}:}
Note that the cohomological vanishing conditions in Theorem \ref{ueda_quotient} follow from  \cite[Theorems 1 and 2]{Mall91}.
Corollary \ref{corB} follows directly from Theorem \ref{ueda_quotient} together with \cite[Theorem 3.3]{SW25}.
\qed
\paragraph{}
Note that when the fundamental group is Abelian with cyclic torsion part, the Dolbeault cohomology with values in a flat line bundle is completely computed in \cite{LZ04}.
In our setting, we only require the vanishing of $H^0$ and $H^1$, which allows us to pass to a finite \'etale cover and treat more general secondary Hopf manifolds than those considered in \cite{LZ04}.
\section{Vertical linearisation}

In this section, we study Ueda's problem under finite \'etale covers.
We begin with the case where $C$ is a smooth hypersurface in $M$.
In this situation, a holomorphic foliation is always induced by a holomorphic form. 
We then turn to the higher-rank case, where we focus on the case of hyperelliptic manifold.
We refer the reader to \cite{Bru15} for further background on holomorphic foliations.
\subsection{Preliminary} For the reader’s convenience, we recall some basic facts about holomorphic foliations.
Let $\hat{C}$ be a complex submanifold of a complex manifold $\hat{M}$.

Recall the definition of holomorphic foliation.
\begin{definition}[Holomorphic foliation]
Let $\hat{M}$ be a complex manifold of dimension $n+d$.
A holomorphic foliation $\cF$ of codimension $r$ on $\hat{M}$ is given by a coherent saturated subsheaf
\[
\cF \subset T_{\hat{M}}
\]
of rank $n+d-r$ that is integrable in the sense of Frobenius, i.e. the Lie bracket of germs of sections of $\cF$ satisfies
\[
[\cF, \cF] \subset \cF.
\]
Here, $[\cF, \cF] \subset \cF$ means that for any point $p \in \hat{M}$ and any germs of holomorphic vector fields $X,Y$ tangent to $\cF$ near $p$, the Lie bracket $[X,Y]$ is also tangent to $\cF$ near $p$.
The subsheaf $\cF$ is said to be \emph{saturated} if $N_{\cF}:=T_{\hat{M}}/\cF$ is torsion-free.

The foliation $\cF$ is said to be \emph{regular} if it is a holomorphic subbundle of $T_{\hat{M}}$.
\end{definition}
\begin{definition}[Foliation induced by a holomorphic form]
\label{form-def}
Let $\omega \in H^0(\hat{M}, \Omega^1_{\hat{M}} \otimes L)$ be a nonzero holomorphic $1$-form with values in a line bundle $L$.
Then $\omega$ defines a codimension-one holomorphic foliation $\cF_\omega$ if the following conditions hold:
\begin{enumerate}
    \item \emph{Regularity:} The zero locus of $\omega$ has codimension at least two.
    The associated coherent subsheaf
    \[
    \cF_\omega := \ker\bigl(
    T_{\hat{M}}
    \xrightarrow{\ \lrcorner_\omega\ }
    L
    \bigr)
    \subset T_{\hat{M}}
    \]
    is therefore a rank \((n+d-1)\) coherent subsheaf of \(T_{\hat{M}}\), and defines a holomorphic subbundle on
    $
    \hat{M} \setminus \{\omega = 0\}.
    $

    \item \emph{Involutivity:} The kernel $\cF_\omega$ is involutive, i.e., for any point $p \in \hat{M}$ and any germs of holomorphic vector fields $X,Y \in \cF_\omega$ near $p$, the Lie bracket $[X,Y]$ also lies in $\cF_\omega$.
\end{enumerate}

In this case, the leaves of $\cF_\omega$ are the maximal connected, locally closed  complex submanifolds $\mathcal{L}$ in the complement of the singular set such that $\omega|_{\mathcal{L}} = 0$.
\end{definition}
It is well known that
a codimension-one regular holomorphic foliation is always defined by a twisted holomorphic $1$-form as in Definition \ref{form-def}. 
\begin{remark}
\label{rem3} Recall that a codimension-one regular holomorphic foliation $\cF \subset T_{\hat{M}}$ is equivalent to a nowhere-vanishing global section \[ \omega \in H^0(\hat{M}, \Omega^1_{\hat{M}} \otimes N_{\cF}), \] satisfying $\omega \wedge d\omega = 0$, where $N_{\cF} := T_{\hat{M}}/\cF$ is the normal bundle of the foliation. The section $\omega$ is unique up to multiplication by a nowhere-vanishing holomorphic function. If $\hat{C}$ is a leaf of the foliation, then $\cF|_{\hat{C}} = T_{\hat{C}}$, yielding a natural identification \[ (N_{\cF})|_{\hat{C}} \simeq N_{\hat{C}/\hat{M}}. \] \end{remark}
We want to study sufficient conditions ensuring that, as germs, a regular holomorphic foliation on $\hat{M}$ with $\hat{C}$ as a compact leaf is unique.
In general, $H^0(\hat{M},\, \Omega^1_{\hat{M}} \otimes N_{\cF})$ is not one-dimensional.
In fact, under suitable vanishing conditions on cohomology groups (e.g., when the normal bundle is non-torsion), one can compute the dimension of $H^0(\hat{M},\, \Omega^1_{\hat{M}} \otimes N_{\cF})$.
\paragraph{}
For Corollary \ref{corC}, recall the definition of a holomorphic foliation induced by a holomorphic form in arbitrary codimension.

\begin{definition}
\label{higher-foliation-def}
Let $\hat{M}$ be a complex manifold of dimension $n+d$, and let
\[
\omega \in H^0(\hat{M}, \Omega^d_{\hat{M}} \otimes L)
\]
be a nonzero holomorphic $d$-form with values in a line bundle $L \to \hat{M}$.

Define the \emph{singular set} of $\omega$ by
\[
\mathrm{Sing}(\omega) := \{ p \in \hat{M} \mid \omega(p) = 0 \}.
\]

We say that $\omega$ \emph{defines a codimension-$d$ holomorphic foliation} $\cF$ on $\hat{M}$ if all of the following conditions hold:

\begin{enumerate}
\item (\textbf{Nondegeneracy on the regular set}) On the open dense set
\[
\hat{M}^{\mathrm{reg}} := \hat{M} \setminus \mathrm{Sing}(\omega),
\]
the kernel distribution
\[
\ker(\omega)_p := \{ v \in T_{\hat{M},p} \mid v \;\lrcorner\; \omega(p) = 0 \}
\]
has constant rank $n$.

\item (\textbf{Integrability}) The kernel distribution $\ker(\omega) \subset T_{\hat{M}}|_{\hat{M}^{\mathrm{reg}}}$ is involutive. That is, for any local holomorphic vector fields $X,Y \in \Gamma(U, \ker(\omega))$, we have $[X,Y] \in \Gamma(U, \ker(\omega))$ on $U \subset \hat{M}^{\mathrm{reg}}$.
Equivalently, $\omega$ satisfies the Frobenius integrability condition:
\[
\forall X_1,\dots,X_{d+1} \in \Gamma(U, T_{\hat{M}}), \quad \omega([X_i,X_j], X_1, \dots, \widehat{X_i}, \dots, \widehat{X_j}, \dots, X_{d+1}) = 0.
\]

\item (\textbf{Singular set codimension}) The singular set $\mathrm{Sing}(\omega)$ has codimension at least $2$ in $\hat{M}$.
\end{enumerate}

In fact, one defines a holomorphic foliation
on $\hat{M}$ as the kernel of the morphism $T_{\hat{M}} \to \Omega^{d-1}_{\hat{M}} \otimes L$ given by the contraction with
$\omega$.
\end{definition}
\subsection{Ueda's problem}
In all this subsection, we always assume that $H^0(C, \cO_C)=\C$.
Since $p$ is finite \'etale, this is equivalent to the condition
that $H^0(\hat{C}, \cO_{\hat{C}})=\C$.
For any connected compact complex manifold $C$, we have $H^0(C, \cO_C)=\C$.

Let $\hat{C}$ be a smooth hypersurface in a complex manifold $\hat{M}$, viewed as a germ.
We want to study regular holomorphic foliations on $M$ having $C$ as a compact leaf.
In other words, we consider Ueda's problem for $(M,C)$.

Throughout this section, we assume that there exists a regular holomorphic foliation $\cF$ on $\hat{M}$ having $\hat{C}$ as a leaf, constructed via vertical linearization (see Setting \ref{set} and Remark \ref{rem2}). The main difficulty is to determine when this foliation descends to a regular holomorphic foliation on $M$ having $C$ as a leaf.

Let $\cI_{\hat{C}} \subset \mathcal{O}_{\hat{M}}$ be the ideal sheaf of $\hat{C}$, and consider the infinitesimal neighborhoods of $\hat{C}$.
A global section of
\[
H^0(\hat{M}, \Omega^1_{\hat{M}} \otimes N_{\cF})
\]
is uniquely determined by its restriction to the infinitesimal neighborhoods of $\hat{C}$.
In other words, the natural map
\[
H^0(\hat{M}, \Omega^1_{\hat{M}}\otimes N_{\cF}) \;\longrightarrow\;
\varprojlim_{k \geq 1} H^0\bigl(\hat{M}, \Omega^1_{\hat{M}} / \cI_{\hat{C}}^{k} \Omega^1_{\hat{M}} \otimes N_{\hat{C}/\hat{M}}\bigr)
\]
is injective by the identity theorem.

The $0$-th infinitesimal neighborhood of $\hat{C}$ in $\hat{M}$ is just the
reduced space $\hat{C}$ itself, equipped with the structure sheaf
$
\mathcal{O}_{\hat{M}} / \cI_{\hat{C}} \;\cong\; \mathcal{O}_{\hat{C}} .
$
Restricting the sheaf of holomorphic $1$-forms $\Omega^1_{\hat{M}}$ (which is a holomorphic vector bundle) to $\hat{C}$, we have the
exact sequence
\begin{equation}
\label{equ2}
0 \;\to\; N^*_{\hat{C}/\hat{M}} \;\to\; \Omega^1_{\hat{M}}|_{\hat{C}}
   \;\to\; \Omega^1_{\hat{C}} \;\to\; 0 ,
\end{equation}
where $N^*_{\hat{C}/\hat{M}} = \cI_{\hat{C}} / \cI_{\hat{C}}^2$ is the conormal bundle of $\hat{C}$ in $\hat{M}$.
From now on, we assume that the restricted exact sequence splits and that
\begin{equation}
\label{cohom1}
H^0(\hat{C},N^{-k}_{\hat{C}/\hat{M}})=H^1(\hat{C},N^{-k}_{\hat{C}/\hat{M}})=0,
\end{equation}
\begin{equation}
\label{cohom2}
H^0(\hat{C},\Omega^1_{\hat{C}} \otimes N^{-k}_{\hat{C}/\hat{M}})=H^1(\hat{C},\Omega^1_{\hat{C}} \otimes N^{-k}_{\hat{C}/\hat{M}})=0,
\end{equation}
for all $k \geq 1$.
Assume furthermore that
\begin{equation}
\label{cohom4}
H^0\!\bigl(\hat{C},\, \Omega^1_{\hat{C}} \otimes N_{\hat{C}/\hat{M}} \bigr) \;=\;0.
\end{equation}

The associated long exact sequence for the restricted short exact sequence, after tensoring with $N_{\hat{C}/\hat{M}}$, is
\[
0 \to H^0\!\big(\hat{C},\cO_{\hat{C}}\big)
  \to H^0\!\big(\hat{C}, \Omega^1_{\hat{M}}|_{\hat{C}} \otimes N_{\hat{C}/\hat{M}}\big)
  \to H^0\!\big(\hat{C}, \Omega^1_{\hat{C}}\otimes N_{\hat{C}/\hat{M}}\big)=0
\]
where the vanishing of the last term follows from our cohomology assumption (\ref{cohom4}).

Let us sketch a solution to Ueda's problem via vertical linearisation in the following remarks.
We will concentrate on the construction of twisted holomorphic forms, from which one can easily verify that the associated holomorphic foliations are defined.
\begin{remark}
\label{rem2}
In this remark, we consider the case where \(\hat{C}\) is a smooth hypersurface in the manifold \(\hat{M}\).

Ueda’s construction of the horizontal foliation via vertical linearization can be interpreted as extending the image of the constant function $1$ under the natural inclusion \begin{equation}
\label{cohom5}
H^0\!\big(\hat{C},\cO_{\hat{C}}\big) \subset H^0\!\big(\hat{C}, \Omega^1_{\hat{M}}|_{\hat{C}} \otimes N_{\hat{C}/\hat{M}}\big)
\end{equation}  
to a neighborhood of $\hat{C}$ in $\hat{M}$. This produces a regular holomorphic foliation having $\hat{C}$ as a leaf. More precisely, consider the local coordinates introduced in Setting \ref{set}. Let $(h_j,v_j)$ be local coordinates on $\hat{M}$ such that $\hat{C}$ is locally defined by $\{v_j=0\}$. Then the differentials $dv_j|_{\hat{C}}$ define local sections in $\Gamma(U_j, \Omega^1_{\hat{M}}|_{\hat{C}} \otimes N_{\hat{C}/\hat{M}})$. Since the normal bundle is Hermitian flat, the transition functions take the form \[ \hat\Phi_{kj} = \bigl(\varphi_{kj}(h_j) + \hat\phi^h_{kj}(h_j,v_j), \, t_{kj} v_j\bigr), \] and therefore the sections $dv_j|_{\hat{C}}$ glue to a global section of $H^0\!\big(\hat{C}, \Omega^1_{\hat{M}}|_{\hat{C}} \otimes N_{\hat{C}/\hat{M}}\big)$. On the other hand, each $dv_j|_{\hat{C}}$ can be viewed as a local generator of $\Gamma(U_j, N_{\hat{C}/\hat{M}}^* \otimes N_{\hat{C}/\hat{M}})=\Gamma(U_j,\cO_{\hat{C}})$, and these generators glue to the constant function $1$. This explains the above inclusion.
Recall that $\hat{C}$ is a smooth deformation retract of $\hat{M}$, which induces an isomorphism on fundamental groups. Since $N_{\hat{C}/\hat{M}}$ is Hermitian flat, it extends naturally to a Hermitian flat line bundle on $\hat{M}$. For simplicity, we keep the same notation $N_{\hat{C}/\hat{M}}$ for this extension. Then the differentials $dv_j$ define local sections in $\Gamma(V_j, \Omega^1_{\hat{M}} \otimes N_{\hat{C}/\hat{M}})$, and these glue together to give a global section in \[ H^0(\hat{M}, \Omega^1_{\hat{M}} \otimes N_{\hat{C}/\hat{M}}). \]
The normal bundle of the foliation defined by this global section coincides with $N_{\hat{C}/\hat{M}}$.
\end{remark}
\begin{remark}
\label{rem4}
Similar arguments apply to the case where $N_{\hat{C}/\hat{M}}$ is a vector bundle which splits as a direct sum of Hermitian flat line bundles. Let $\{V_j\}$ be a covering of $\hat{M}$ by charts, and let $v_j$ denote the vertical coordinate on $V_j$. The differential $dv_j$ is then a holomorphic $1$-form on $V_j$ with values in $N_{\hat{C}/\hat{M}}$. Since the transition functions of $N_{\hat{C}/\hat{M}}$ are unitary and diagonal, the change of vertical coordinates is linear and compatible with the flat structure. It follows that the collection $\{dv_j\}$ defines a global section of $\Omega^1_{\hat{M}} \otimes N_{\hat{C}/\hat{M}}$. Moreover, the components obtained in this way are linearly independent along $\hat{C}$. Consequently, \[ h^0(\hat{M}, \Omega^1_{\hat{M}} \otimes N_{\hat{C}/\hat{M}}) \geq \mathrm{rank}(N_{\hat{C}/\hat{M}}). \]
\end{remark}

Let us first consider the case where $N_{\hat{C}/\hat{M}}$ is a line bundle to prove Theorem \ref{ueda_quotient}.

The higher order $k$-th infinitesimal neighborhood of $\hat{C}$ is defined by $\cI_{\hat{C}}^{k+1}$. There is a
short exact sequence:
\[
0 \longrightarrow \cI_{\hat{C}}^k / \cI_{\hat{C}}^{k+1} \longrightarrow \mathcal{O}_{\hat{M}} / \cI_{\hat{C}}^{k+1}
\longrightarrow \mathcal{O}_{\hat{M}} / \cI_{\hat{C}}^k \longrightarrow 0.
\]

Tensoring with $\Omega^1_{\hat{M}} \otimes N_{\cF}$ and restricting to $\hat{C}$, we get a short exact sequence of coherent sheaves
\[
0 \to \Omega^1_{\hat{M}} \otimes N_{\cF}|_{\hat{C}} \otimes N_{\hat{C}/\hat{M}}^{-k} \to
\Omega^1_{\hat{M}} \otimes N_{\cF} \otimes \mathcal{O}_{\hat{M}} / \cI_{\hat{C}}^{k+1} \to
\Omega^1_{\hat{M}} \otimes N_{\cF} \otimes \mathcal{O}_{\hat{M}} / \cI_{\hat{C}}^k \to 0,
\]
since $\cI_{\hat{C}}^k / \cI_{\hat{C}}^{k+1} \simeq N_{\hat{C}/\hat{M}}^{-k}$.

Taking global sections, we obtain the exact sequence
\[
0 \to H^0(\hat{C}, \Omega^1_{\hat{M}}|_{\hat{C}} \otimes N_{\hat{C}/\hat{M}}^{1-k})
\to H^0(\hat{M}, \Omega^1_{\hat{M}} / \cI_{\hat{C}}^{k+1} \Omega^1_{\hat{M}}\otimes N_{\hat{C}/\hat{M}})
\to \]
\[
H^0(\hat{M}, \Omega^1_{\hat{M}} / \cI_{\hat{C}}^{k} \Omega^1_{\hat{M}}\otimes N_{\hat{C}/\hat{M}}) \to H^1(\hat{C}, \Omega^1_{\hat{M}}|_{\hat{C}} \otimes N_{\hat{C}/\hat{M}}^{1-k}).
\]
Note that
\[ (N_{\cF})|_{\hat{C}} \simeq N_{\hat{C}/\hat{M}}. \]

By our cohomology assumptions (\ref{cohom1}) and the splitting of short exact sequence (\ref{equ2}), for all $k \geq 2$,
$$ H^1(\hat{C}, \Omega^1_{\hat{M}}|_{\hat{C}} \otimes N_{\hat{C}/\hat{M}}^{1-k})= H^0(\hat{C}, \Omega^1_{\hat{M}}|_{\hat{C}} \otimes N_{\hat{C}/\hat{M}}^{1-k})=0$$
which implies that
$$
H^0(\hat{M}, \Omega^1_{\hat{M}} / \cI_{\hat{C}}^{k+1} \Omega^1_{\hat{M}}\otimes N_{\hat{C}/\hat{M}})
\simeq H^0(\hat{M}, \Omega^1_{\hat{M}} / \cI_{\hat{C}}^{k} \Omega^1_{\hat{M}}\otimes N_{\hat{C}/\hat{M}}).
$$
For $k=1$, since  from our cohomology assumption (\ref{cohom4}) and the splitting of short exact sequence (\ref{equ2}),
\[
h^0\!\big(\hat{C}, \Omega^1_{\hat{M}}|_{\hat{C}} \otimes N_{\hat{C}/\hat{M}}\big) = 1.
\]
Assume that a vertical linearisation exists near \(\hat{C}\), and that the foliation \(\cF\) is defined as in Remark \ref{rem2}.
Then by Remark \ref{rem2} (in particular (\ref{cohom5})), the factorization
\[
H^0(\hat{M}, \Omega^1_{\hat{M}} \otimes N_{\cF})
\;\longrightarrow\;
H^0\big(\hat{M}, \Omega^1_{\hat{M}} / \cI_{\hat{C}}^{2} \Omega^1_{\hat{M}} \otimes N_{\hat{C}/\hat{M}}\big)
\;\longrightarrow\;
H^0\big(\hat{M}, \Omega^1_{\hat{M}} / \cI_{\hat{C}} \Omega^1_{\hat{M}} \otimes N_{\hat{C}/\hat{M}}\big)
\]
implies that the map
\[
H^0\big(\hat{M}, \Omega^1_{\hat{M}} / \cI_{\hat{C}}^{2} \Omega^1_{\hat{M}} \otimes N_{\hat{C}/\hat{M}}\big)
\;\longrightarrow\;
H^0\big(\hat{M}, \Omega^1_{\hat{M}} / \cI_{\hat{C}} \Omega^1_{\hat{M}} \otimes N_{\hat{C}/\hat{M}}\big)
\]
is surjective.

Note that, by our cohomology assumptions (\ref{cohom1}),
\[
H^0(\hat{C}, \Omega^1_{\hat{M}}|_{\hat{C}}) = H^0(\hat{C}, \Omega^1_{\hat{C}})
\]
which is the kernel of 
\[
H^0\big(\hat{M}, \Omega^1_{\hat{M}} / \cI_{\hat{C}}^{2} \Omega^1_{\hat{M}} \otimes N_{\hat{C}/\hat{M}}\big)
\;\longrightarrow\;
H^0\big(\hat{M}, \Omega^1_{\hat{M}} / \cI_{\hat{C}} \Omega^1_{\hat{M}} \otimes N_{\hat{C}/\hat{M}}\big).
\]

Moreover,
since $H^0\big(\hat{C}, (\Omega^1_{\hat{M}} \otimes N_{\cF})|_{\hat{C}}\big) \simeq \C$,
the restriction map
\[
H^0(\hat{M}, \Omega^1_{\hat{M}} \otimes N_{\cF}) \;\longrightarrow\;
H^0\big(\hat{C}, (\Omega^1_{\hat{M}} \otimes N_{\cF})|_{\hat{C}}\big)
\]
is surjective by Remark \ref{rem2}.
In conclusion, $H^0(\hat{M}, \Omega^1_{\hat{M}} \otimes N_{\cF})$ has dimension $h^{1,0}(\hat{C}) + 1$.

We emphasize that, in general, $H^0(\hat{M}, \Omega^1_{\hat{M}} \otimes N_{\cF})$ is not 1 dimensional.
In particular, the corresponding holomorphic form may not be invariant under the action of the Galois group $G$.
For our purposes, we aim to construct a $G$-invariant nowhere-vanishing section in $H^0(\hat{M}, \Omega^1_{\hat{M}} \otimes N_{\cF})$. Such a section descends to a nowhere-vanishing section on $M$, and hence, by Remark \ref{rem3}, is equivalent to defining a regular holomorphic foliation on $M$.

Our observation is that  the kernel of the natural map
$$
H^0(\hat{M}, \Omega^1_{\hat{M}}\otimes N_{\cF}) \to
\varprojlim_{k \geq 2}   H^0(\hat{M}, \Omega^1_{\hat{M}} / \cI_{\hat{C}}^{k} \Omega^1_{\hat{M}}\otimes N_{\hat{C}/\hat{M}})
$$
which is 
\[
H^0\!\big(\hat{C}, \Omega^1_{\hat{M}}|_{\hat{C}} \otimes N_{\hat{C}/\hat{M}}\big),
\]
is of dimension 1.
Using this, we can give a sufficient condition for the existence of a foliation on $M$ induced from a $G$-invariant foliation on $\hat{M}$.
\paragraph{}
\textit{Proof of Theorem \ref{ueda-quotient2}:}
It is enough to show that Ueda's construction of the foliation on $\hat{M}$, defining a global section of
\[
H^0(\hat{M}, \Omega^1_{\hat{M}} \otimes N_{\hat{C}/\hat{M}})
\]
as in Remark \ref{rem2}
descends to a global section of
\[
s \in H^0(M, \Omega^1_{M} \otimes N_{C/M}).
\]
Before proving this, we explain why it is sufficient to complete the proof of Theorem \ref{ueda-quotient2}.
Since $p$ is finite \'etale, the descended section is nowhere-vanishing. By Remark \ref{rem3}, it therefore defines a regular holomorphic foliation on $M$. Moreover, since $p^*s|_{\hat{C}}=0$, we obtain $s|_{C}=0$. Hence $C$ is invariant and is contained in a leaf $\mathcal{L}$. Because the foliation is regular, $\mathcal{L}$ is a locally closed complex submanifold. Since $C$ is a complex submanifold of the same dimension as $L$ and is tangent to the foliation, the inclusion $C \subset \mathcal{L}$ is locally biholomorphic. Therefore, $C$ is an open subset of $\mathcal{L}$. To conclude the equality, we pass to the finite étale cover.
Since $p$ is \'etale, $p^{-1}(\mathcal{L})$ is locally closed submanifold of $\hat{M}$. It is invariant in $\hat{M}$ and hence each connected component of $p^{-1}(\mathcal{L})$ is contained in some leaf of $\hat{M}$.
Since $C$ is an open subset of $\mathcal{L}$, we have $\hat{C}=p^{-1}(C)$, which is an open subset of $p^{-1}(\mathcal{L})$. In particular, $\hat{C}$ is contained in a connected component of $p^{-1}(\mathcal{L})$ and is open in it. As this component is contained in a leaf and $\hat{C}$ itself is a leaf, it follows that this component coincides with $\hat{C}$.
Since the covering is Galois, any other connected component is obtained from this one by the action of some element of the Galois group.
Since $\hat{C}$ is invariant under the action of the Galois group,
$p^{-1}(\mathcal{L})=\hat{C}$.
We conclude that $C=\mathcal{L}$.
In other words, we have shown that $C$ is a leaf of the induced foliation.
The regularity and involutivity are stable under finite \'etale cover,
thus the section $s$ defines a regular holomorphic foliation on $M$ and this completes the proof of Theorem \ref{ueda-quotient2}.

Consider the kernel of the map
\[
H^0(\hat{M}, \Omega^1_{\hat{M}}\otimes N_{\hat{C}/\hat{M}}) \;\longrightarrow\;
\varprojlim_{k \geq 2} H^0\big(\hat{M}, \Omega^1_{\hat{M}} / \cI_{\hat{C}}^{k} \Omega^1_{\hat{M}} \otimes N_{\hat{C}/\hat{M}}\big),
\]
with generator
\[
\sigma \in
H^0\!\big(\hat{C}, \Omega^1_{\hat{M}}|_{\hat{C}} \otimes N_{\hat{C}/\hat{M}}\big).
\]

Since both sides are $G$-invariant, for any $g \in G$, $g^* \sigma$ is also in the kernel and, by the dimension assumption, is a multiple of $\sigma$.
Thus there exists a complex character $\chi \in \hat{G}$ such that
\[
g^* \sigma = \chi(g) \sigma.
\]
In fact, for any $g_1, g_2 \in G$, we have \[ g_2^* g_1^* \sigma = \chi(g_1)\chi(g_2)\sigma = (g_1 g_2)^* \sigma = \chi(g_1 g_2)\sigma. \] This shows that \[ \chi(g_1)\chi(g_2) = \chi(g_1 g_2). \]

Now by our cohomology assumption (\ref{cohom4}) and the splitting of short exact sequence (\ref{equ2}), consider 
\[
\sigma|_{\hat{C}} \in H^0(\hat{M}, \Omega^1_{\hat{M}}|_{\hat{C}} \otimes N_{\hat{C}/\hat{M}}) = H^0(\hat{C}, \cO_{\hat{C}})=\C,
\]
which implies that $\chi$ is the trivial character.
In other words, $\sigma$ is $G$-invariant and thus descends to a nontrivial global section of
\[
s\in H^0(M, \Omega^1_{M} \otimes N_{C/M}).
\]
\qed
\paragraph{}
Since full linearization implies vertical linearization, we will not discuss Ueda's problem for secondary Hopf manifolds (see Corollary \ref{corB}).
However, as an application, we can study Ueda's problem for hyperelliptic manifolds as follows.
\begin{corollary}
\label{cor1}
Let $C$ be a hyperelliptic manifold of dimension $n$
embedded in a complex
manifold $M$ of dimension $n + 1$.  Suppose that the normal bundle of $C$ in $M$ admits transition functions that are locally constant Hermitian matrices.
Let $p:\hat{C} \to C$  be a finite \'etale Galois cover such that $\hat{C}$ is a complex torus and assume that the induced normal bundle $N_{\hat{C}/\hat{M}}$ satisfies the strongly vertically Diophantine condition (see Definition \ref{dioph}).
Then a neighborhood of $C$ in $M$ admits a regular holomorphic foliation with $C$ as a leaf.
\end{corollary}
\begin{proof}
 Note that the strongly vertically Diophantine condition (more precisely, the non resonance condition) implies that the normal line bundle
 $N_{\hat{C}/\hat{M}}$
is non-torsion.
Note that the tangent or cotangent bundle of $\hat{C}$ is trivial.
This implies by Lemma \ref{cohom_torus} that the short exact sequence
$$
0 \;\to\; N^*_{\hat{C}/\hat{M}} \;\to\; \Omega^1_{\hat{M}}|_{\hat{C}}
   \;\to\; \Omega^1_{\hat{C}} \;\to\; 0 ,
$$ splits and that
$$
H^0(\hat{C},N^{-k}_{\hat{C}/\hat{M}})=H^1(\hat{C},N^{-k}_{\hat{C}/\hat{M}})=0,
$$
$$
H^0(\hat{C},\Omega^1_{\hat{C}} \otimes N^{-k}_{\hat{C}/\hat{M}})=H^1(\hat{C},\Omega^1_{\hat{C}} \otimes N^{-k}_{\hat{C}/\hat{M}})=0,
$$
for all $k \geq 1$.
Moreover, Lemma \ref{cohom_torus} also implies that
$$
H^0\!\bigl(\hat{C},\, \Omega^1_{\hat{C}} \otimes N_{\hat{C}/\hat{M}} \bigr) \;=\;
0.
$$
The conclusion follows from Theorem \ref{ueda-quotient2} and \cite[Theorem 1.1]{SW24}.
We remark that the proof of \cite[Theorem 1.1]{SW24} is obtained via the vertical linearisation stated in the assumption of Theorem \ref{ueda-quotient2}, which is stronger than the conclusion that Ueda's problem admits a solution.
\end{proof}


%
In general, when the normal bundle is a Hermitian flat vector bundle, there is no explicit general formula for the cohomology with coefficients in this bundle. To simplify the presentation, we restrict to the case where $\hat{C}$ is a complex torus and the normal bundle splits as a direct sum of Hermitian flat line bundles.
\paragraph{}
\textit{Proof of Corollary \ref{corC}:}
As before, consider the kernel of
$$
H^0(\hat{M}, \Omega^1_{\hat{M}}\otimes N_{\hat{C}/\hat{M}}) \to
\varprojlim_{k \geq 2}   H^0(\hat{M}, \Omega^1_{\hat{M}} / \cI_{\hat{C}}^{k} \Omega^1_{\hat{M}}\otimes N_{\hat{C}/\hat{M}}).
$$
Denote this kernel by $N$.

In the higher-codimension case, the normal bundle $N_{\hat{C}/\hat{M}}$ (as a Hermitian flat vector bundle over a complex torus) is a direct sum of Hermitian flat line bundles.
(Note that the normal bundle  $N_{C/M}$ is not necessarily a direct sum of Hermitian flat line bundles.)
The strongly vertically Diophantine condition \ref{dioph} (more precisely, the non resonance condition) then implies that
$$\Omega^1_{\hat{M}}|_{\hat{C}} \simeq  N_{\hat{C}/\hat{M}}^* \oplus  \Omega^1_{\hat{C}}.$$
Moreover, we have
\[
H^0\!\bigl(\hat{C},\,\Omega^1_{\hat{M}}|_{\hat{C}} \otimes N_{\hat{C}/\hat{M}}  \otimes \Sym^k N_{\hat{C}/\hat{M}}^\ast \bigr) \;=\;
H^1\!\bigl(\hat{C},\, \Omega^1_{\hat{M}}|_{\hat{C}} \otimes N_{\hat{C}/\hat{M}}  \otimes\Sym^k N_{\hat{C}/\hat{M}}^\ast \bigr) \;=\; 0,
\]
for all $k \geq 2$.
The same cohomological calculation as above works for the \(k\)-th infinitesimal neighborhood of \(\hat{C}\) for all \(k \geq 2\).

We claim that the restriction
$$
H^0(\hat{M}, \Omega^1_{\hat{M}}\otimes N_{\hat{C}/\hat{M}}) \to
   H^0(\hat{M}, \Omega^1_{\hat{M}} / \cI_{\hat{C}}^{} \Omega^1_{\hat{M}}\otimes N_{\hat{C}/\hat{M}})
$$
is surjective.
Note that
$$ H^0(\hat{M}, \Omega^1_{\hat{M}} / \cI_{\hat{C}}^{} \Omega^1_{\hat{M}}\otimes N_{\hat{C}/\hat{M}}) \simeq  H^0(\hat{C}, \Omega^1_{\hat{M}}|_{\hat{C}} \otimes N_{\hat{C}/\hat{M}}).$$
In fact,
the strongly vertically Diophantine condition \ref{dioph} (more precisely, the non resonance condition) then implies that
\[
H^0\big(\hat{C}, \Omega^1_{\hat{M}}|_{\hat{C}} \otimes N_{\hat{C}/\hat{M}}\big) = \C^d.
\]
By Remark \ref{rem4},  the restriction
$$
H^0(\hat{M}, \Omega^1_{\hat{M}}\otimes N_{\hat{C}/\hat{M}}) \to
   H^0(\hat{M}, \Omega^1_{\hat{M}} / \cI_{\hat{C}}^{} \Omega^1_{\hat{M}}\otimes N_{\hat{C}/\hat{M}})
$$
is hence surjective.

In conclusion, by the same cohomological calculation as in the proof of Theorem \ref{ueda-quotient2} presented above,  the restriction map induces a linear isomorphism \[ N \longrightarrow H^0(\hat{C}, N^*_{\hat{C}/\hat{M}} \otimes N_{\hat{C}/\hat{M}}) \subset H^0(\hat{C}, \Omega^1_{\hat{M}}|_{\hat{C}} \otimes N_{\hat{C}/\hat{M}}), \] whose dimension is $d$ in our situation. Moreover, this isomorphism is $G$-equivariant. On the other hand, the $d$ holomorphic $1$-forms obtained Remark \ref{rem4} are not necessarily $G$-invariant.

Observe that the wedge product of these holomorphic forms defines a global section of
\[
H^0(\hat{M}, \Omega^d_{\hat{M}}\otimes \det(N_{\hat{C}/\hat{M}})).
\]
We use this section to construct a foliation as in Definition \ref{higher-foliation-def}.

In local coordinates, one can easily check that this wedge product defines a regular codimension-$d$ holomorphic foliation.
We claim that it is $G$-invariant.
Thus it descends to a global section of
\[
H^0(M, \Omega^d_M \otimes \det(N_{C/M})),
\]
which defines a foliation as in Definition \ref{higher-foliation-def}.
This foliation is regular and has $C$ as a leaf. This follows from the same argument as in the proof of Corollary \ref{cor1}.

The $G$-invariance follows from the following alternative interpretation of this section.
Since $N^*_{\hat{C}/\hat{M}}\otimes N_{\hat{C}/\hat{M}}$ is isomorphic to
\[
\mathrm{Hom}(N_{\hat{C}/\hat{M}}, N_{\hat{C}/\hat{M}}),
\]
the wedge product can be regarded as a global section
\[
H^0\big(\hat{C}, \mathrm{Hom}(\det(N_{\hat{C}/\hat{M}}), \det(N_{\hat{C}/\hat{M}}))\big) = H^0(\hat{C}, \cO_{\hat{C}}),
\]
in which every global section is $G$-invariant.
\qed
 
Department of Mathematics, Graduate School of Science, Osaka Metropolitan University, 3-3-138 Sugimoto, Osaka 558-8585, Japan.\\
Email address: xiaojun.wu@univ-cotedazur.fr; y25161q@omu.ac.jp.  
\end{document}